\documentclass{amsart}
\newtheorem{theorem}{Theorem}[section]
\newtheorem{corollary}[theorem]{Corollary}
\newtheorem{problem}[theorem]{Problem}

\newtheorem{proposition}[theorem]{Proposition}
\newtheorem{remark}[theorem]{Remark}
\title[On the first passage time of a continuous Martingale]{On the first passage time density of a continuous Martingale over a moving boundary}
\author{Gerardo Hernandez-del-Valle}
\address{Statistics Department, Columbia University\\Mail Code 4403, New York, N.Y.}\email{gerardo@stat.columbia.edu}
\date{March 2008}
\begin{document}
\maketitle
\begin{abstract}

In this paper we derive the density $\varphi$ of the first time $T$ that a continuous martingale $M$ with non-random quadratic variation $\langle M\rangle_\cdot:=\int_0^\cdot h^2(u)du$ hits a moving boundary $f$ which is twice continuously differentiable, and $f'/h\in\mathbb{C}^2[0,\infty)$. 

Thus, this work is an extension to case in which $M$ is in fact a one-dimensional standard Brownian motion $B$, as studied in Hernandez-del-Valle (2007).
\end{abstract}
\section{Introduction}
 Patie (2004) [Section 3.6.1], for instance, has pointed out the existing relationship between first passage time problems and 3-dimensional Bessel bridges. In Hernandez-del-Valle (2007), the author uses this observation to derive the density of the first time $T$ that a one-dimensional, standard Brownian motion $B$ reaches the twice continuously differentiable boundary $f$. 

The idea is to first transform the first-passage time into a Cauchy problem, via Girsanov's and Feynman-Kac's theorems [as has been done before, for instance, by: Groeneboom (1987), Martin-L\"of (1998), Novikov (1981), and Salminen (1988)]. Next, after properly manipulating the corresponding {\it parabolic\/} partial differential equation, it is possible to obtain Schr\"odinger's equation for time-dependent linear potential which has been solved analytically by Feng (2001).

In this paper we derive the density $\varphi$ of the first time that a continuous martingale $M$ with non-random quadratic variation $\langle M\rangle_\cdot:=\int_0^\cdot h^2(u)du$ reaches a moving bounday $f$, {\it i.e.\/}
\begin{eqnarray*}
T:=\inf\left\{t\geq 0|M_t=f(t)\right\},
\end{eqnarray*}
provided that the following two conditions hold:
\begin{eqnarray*}
f\in\mathbb{C}^2[0,\infty)\quad\hbox{and}\quad f'/h\in\mathbb{C}^2[0,\infty).
\end{eqnarray*}
The idea, once more, is to first derive and next solve a Cauchy problem, which in this case, is of the following form: 
\begin{eqnarray}
\nonumber-\frac{\partial\varphi}{\partial t}(t,a)&=&-\beta'(t)a\varphi(t,a)+\frac{1}{2}h^2(t)\frac{\partial\varphi}{\partial a^2}(t,a)\\ 
\label{cc}&&+h^2(t)\left(\frac{1}{a}-\frac{a}{\int_t^sh^2(u)du}\right)\frac{\partial\varphi}{\partial a}(t,a)\quad [0,s)\times(0,\infty)\\
\nonumber v(s,a)&=&1,\enskip\hbox{where}\\
\nonumber\beta(u)&:=&\frac{f'(u)}{h^2(u)}\qquad u\geq 0.
\end{eqnarray}

To understand better the application of this result, let us start with an example: Suppose we want to find the density of first time $T$ that a mean-reverting Ornstein-Uhlenbeck process $X:=\{X_t,t\geq 0\}$ starting at zero, and with unit diffusion coefficient will hit a moving boundary $f$, {\it i.e.\/}
\begin{equation}\label{or}
T:=\inf\left\{t\geq 0|X_t=f(t)\right\}.
\end{equation}
It is the solution to this type of problems that we tackle in this paper, since equation (\ref{or}) is equivalent to: 
\begin{equation*}
T:=\inf\left\{t\geq 0\Bigg{|}\int_0^te^udB_u=f(t)e^t\right\}
\end{equation*}
where $B$ is a Brownian motion. 

Furthermore, observe that this type of results may in principle be used for the pricing of {\it moving\/} barrier options, for instance, Lo and Hui (2006) study the case in which the driving process is a constant elasticity of variance (CEV) process. For the case in which the barrier is constant and the process is geometric Brownian motion, the reader may consult Karatzas (1997). 

The paper is organized as follows: In Section \ref{2} we introduce the {\it problem\/} and state our main results, which we will pursue in the remainder of the paper. Next, in Section \ref{3} we introduce some auxiliary results, namely: (1) the density $\varphi_a$ of the first time $T_a$ that  a continuous martingale $\tilde{M}$ starting at zero reaches level $a>0$, (2) the transition density of $\tilde{M}$ or its associated Green's function, and (3) the transition density of $\tilde{M}$ conditioned to reach level $a$ for the first time at time $s$ [see for instance: Imhof (1984) or Revuz and Yor (2005)].

We devote Section \ref{4} to verify that the dynamics  of the conditioned process $(\tilde{M}|T_a=s)$ indeed satisfies the transition density derived in Section \ref{3} by obtaining, once again, its Green's function. Finally, after deriving the corresponding Cauchy problem (\ref{cc}) in Section \ref{5}, we solve it in Section \ref{6}.
\begin{remark}
 For brevity we have not explicitly stated Feynman-Kac's theorem, we thus suggest the reader to consult for instance: Theorem 5.7.6, and Definition 5.7.9 in Karatzas and Shreve (1991). Finally, we would like to mention that  the tools used in addressing
  the parabolic p.d.e's studied in this paper,  are very much in the spirit of Takamizawa and Shoji (2003), and Feng (2001).
\end{remark}

\section{The problem and main results}\label{2}

For expository purposes, we present our main results in this section, and devote the remainder of the paper to their proof.

\begin{problem}\label{prob1} The main motivation of this work is finding the density $\varphi_T$ of the first time $T$, that a one-dimensional, continuous martingale $M$ with non-random quadratic variation
\begin{equation*}
\langle M\rangle_t:=\int_0^th^2(u)du,\quad t\geq 0,
\end{equation*}
 reaches the twice continuously differentiable moving boundary $f$:
\begin{equation}\label{eq}
T:=\inf\left\{t\geq 0|M_t=f(t)\right\},
\end{equation}
given the following two constraints:
\begin{eqnarray*}
f\in\mathbb{C}^2[0,\infty)\quad\hbox{and}\quad f'/h\in\mathbb{C}^2[0,\infty).
\end{eqnarray*}
Namely, the continuous martingale $M$ has the following dynamics:
\begin{equation*}
dM_t=h(t)dB_t,\qquad M_0=0,\enskip t\geq 0
\end{equation*}
where $B$ is a one-dimensional standard Brownian motion under some measure $\mathbb{P}$.
\end{problem}
Alternatively, from an application of Girsanov's theorem and the forthcoming careful construction (in Section \ref{4}) of a Markov process $\tilde{Y}$ which starts at level $a$ and reaches for the first time zero at time $s$ ({\it i.e.\/} $\tilde{Y}$ is a Bessel bridge type process), it follows that Problem \ref{prob1} may be restated as:
\begin{theorem}\label{prop1}
Given that: $(i)$
\begin{equation}\label{cons}
\varphi_a(t)=\frac{ah^2(t)}{\sqrt{2\pi\left(\int_0^th^2(u)du\right)^3}}\exp\left\{-\frac{a^2}{2 \int_0^t h^2(u)du}\right\}
\end{equation}
is the density of the first  time  that a continuous martingale $M$, with quadratic variation $\langle M\rangle_\cdot=\int_0^\cdot h^2(u)du$, reaches the fixed level $a\in(0,\infty)$. $(ii)$ The process $\tilde{Y}$ has the following dynamics:
\begin{equation}\label{di}
d\tilde{Y}_t=h(t)dW_t+h^2(t)\left(\frac{1}{\tilde{Y}_t}-\frac{\tilde{Y}_t}{\int_t^sh^2(u)du}\right)dt,\quad \tilde{Y}_0=a\in(0,\infty),\enskip t\in[0,s]
\end{equation}
(where $W$ is a Wiener process). Then the distribution of $T$, defined in equation (\ref{eq}), equals
\begin{eqnarray}
\nonumber\mathbb{P}(T<t)&=&\int_0^t\tilde{\mathbb{E}}\left[\exp\left\{-\int_0^s\beta'(u)\tilde{Y}_udu\right\}\right]\\
\label{eq2}&&\times\exp\left\{-\beta(s)a+a\int_0^s\beta'(u)du-\frac{1}{2}\int_0^s b(u)\beta(u)du\right\}\varphi_a(s)ds
\end{eqnarray}
provided that both: $f\in\mathbb{C}^2[0,\infty)$ and $f'/h\in \mathbb{C}^2[0,\infty)$ and given that
$\beta(u):=b(u)/h^2(u)$, and $\,b(u):=f'(u)$ for $u\in\mathbb{R}^+$. 
\end{theorem}
From (\ref{eq2}) it becomes clear, that  our  goal is  thus to compute the following expected value:
\begin{equation}\label{hec}
\tilde{\mathbb{E}}\left[\exp\left\{-\int_0^s\beta'(u)\tilde{Y}_udu\right\}\right].
\end{equation}

We shall prove, in Section \ref{6},  that in fact equation (\ref{hec})  has the following solution:
\begin{theorem} Let: (i) $f\in\mathbb{C}^2[0,\infty)$ and $f/h\in\mathbb{C}^2[0,\infty)$, (ii) $\beta'(u)\geq 0$ for all $u>0$ and, (iii) the process $\tilde{Y}$ is as in  equation (\ref{di}) in Theorem \ref{prop1}. Then
\begin{equation*}
\tilde{\mathbb{E}}^{t,a}\left[\exp\left\{-\int_t^s\beta'(u)\tilde{Y}_udu\right\}\right]=\int_0^\infty G(t,a;s,b)db.
\end{equation*}
where:\\[0.1cm]
Given that
\begin{eqnarray*}
\tilde{G}(t',y;\tau',z):=\frac{1}{\sqrt{2\pi(\tau'-t')}}\left[\exp\left\{-\frac{(z-y)^2}{2(\tau'-t')}\right\}-\exp\left\{-\frac{(z+y)^2}{2(\tau'-t')}\right\}\right]
\end{eqnarray*}
[for $0\leq t<\tau\leq s$ and $a,b\in(0,\infty)$], and
$\tau'-t:'=\int_t^\tau h^2(u)du$, $y:=a+v(t)$, $z:=b+\tilde{v}(\tau)$ are obtained by solving:
\begin{eqnarray*}
&&\pi_t(t)=\beta'(t)\enskip\qquad\quad
v'(t)=-h^2(t)\pi(t)\\
&&\tilde{\pi}_\tau(\tau)=-\tilde{\beta}'(\tau)\qquad \tilde{v}'(\tau)=h^2(\tau)\tilde{\pi}(\tau).
\end{eqnarray*} 
Then the Green's function $G$ equals:
\begin{equation*}
G(t,a;\tau,z)= \exp\left\{b\tilde{\pi}(\tau)+a\pi(t)+\frac{1}{2}\int_t^\tau h^2(u)\pi^2(u)\right\}\tilde{G}(t',y;\tau',z).
\end{equation*}

\end{theorem}

Thus: 
\begin{corollary} The density $\varphi_T$ of the stopping time $T$ defined in equation (\ref{eq}) equals:
\begin{eqnarray*}
\varphi_T(s)&=&\int_0^\infty G(0,a;s,b)db\\
&&\times\exp\left\{-\beta(s)a+a\int_0^s\beta'(u)du-\frac{1}{2}\int_0^s b(u)\beta(u)du\right\}\varphi_a(s),\quad s\geq 0.
\end{eqnarray*}
and $\varphi_a$ is defined in (\ref{cons}).
\end{corollary}
\section{Auxiliary results}\label{3}
This section is devoted to  the study and derivation of the transition probability of a martingale $M$ conditioned to reach for the first time some fixed level $a>0$ at precisely time $s>0$. The reader should observe that this analysis will lead to the construction of Bessel bridge type process.

We first begin with a couple of well known ancillary results:
\subsection{Density of the first passage time of $\tilde{M}$ over a fixed boundary $a$}
Observe that if:
\begin{eqnarray}\label{mart}
\tilde{M}_t:=\int_0^th(u)d\tilde{B}_u,\quad\tilde{M}_0=0,\enskip t\geq 0
\end{eqnarray}
is a $\tilde{\mathbb{P}}$-martingale and we define the stopping time $T_a$ as
\begin{equation}\label{ta}
T_a:=\inf\left\{s\geq 0|\tilde{M}_s=a\right\}\qquad a>0.\\
\end{equation}
Then,
the density $\varphi_a$ of $T_a$ under $\tilde{\mathbb{P}}$ is:
\begin{equation}\label{denpa}
\varphi_a(t)=\frac{ah^2(t)}{\sqrt{2\pi\left(\int_0^th^2(u)du\right)^3}}\exp\left\{-\frac{a^2}{2 \int_0^t h^2(u)du}\right\}
\end{equation}
for $a>0$ and $t\geq 0$.\\[0.2cm]
Alternatively,
\subsection{Distribution of $\tilde{M}$}
From Feynman-Kac's theorem 
it follows that the transition probability  of $\tilde{M}$, equation (\ref{mart}),  is
\begin{equation}\label{gg1}
G(t,x;s,y)=\frac{1}{\sqrt{2\pi\left(\int_t^sh^2(u)du\right)}}\exp\left\{-\frac{(y-x)^2}{2\left(\int^s_th^2(u)du\right)}\right\}
\end{equation}
for $t\geq s$.
\subsection{Derivation of the transition probability of the conditioned $\tilde{\mathbb{P}}$ martingale $\tilde{Y}$} [The techniques employed in this section are very much in the spirit of Imhof (1984), but may also be found in Revuz and Yor (2005).] Let $\tilde{M}=a-\tilde{Y}$ be a time changed $\tilde{\mathbb{P}}$-Brownian motion which starts at $0$.

 For $x,a\in\mathbb{R}^+$ we introduce the transition densities defined over $\mathbb{R}^+$ by
\begin{eqnarray}\label{nose1}
\nonumber p(t_0,a;t,y)&:=&\tilde{\mathbb{P}}^a(\tilde{Y}_t\in dy)/dy,\\
\nonumber &=&G(t_0,a;t,y)\\
p_x(t_0,a;t,y)&:=&\tilde{\mathbb{P}}^a(\tilde{Y}_t\in dy,T_x>t)/dy,
\end{eqnarray}
where $G$ and $T_x$ are  as in equations (\ref{gg1}), and   (\ref{ta}) respectively.
From the reflection principle, it follows that
\begin{eqnarray*}
\tilde{\mathbb{P}}^a(\tilde{Y}_t\in dy,T_x>t)&:=&\tilde{\mathbb{P}}^a(Y_t\in dy)-\tilde{\mathbb{P}}^a(Y_t\in dy,T_x<t)\\
&=&\tilde{\mathbb{P}}^a(\tilde{Y}_t\in dy)-\mathbb{P}^a(Y_t\in d(2x-y))\\
&=&G(t_0,a;t,y)dy-G(t_0,a;t,2x-y)dy;
\end{eqnarray*}
hence
\begin{eqnarray*}
p_x(t_0,a;t,y)&=&p(t_0,a;t,y)-p(t_0,a;t,2x-y)\\
&=&G(t_0,a;t,y)-G(t_0,a;t,2x-y).
\end{eqnarray*}
Furthermore, recall from equation (\ref{denpa}) the density of the stopping time $T_x$, and let 
\begin{eqnarray}\label{nose}
f(t_0,a;t,x)&:=&\tilde{\mathbb{P}}^a(T_x\in dt)/dt\\
\nonumber&=&\frac{h^2(t)}{\sqrt{2\pi(\int_{t_0}^th^2(u)du)^3}}\exp\left\{-\frac{(a-x)^2}{2\int_{t_0}^th^2(u)du}\right\}.
\end{eqnarray}
If we define
\begin{equation*}
\tilde{\mathbb{P}}^{a,x;s}(\cdot)=\tilde{\mathbb{P}}^a(\cdot|T_x=s),
\end{equation*}
then
\begin{eqnarray*}
\tilde{\mathbb{P}}^{a,x;s}(\tilde{Y}_{t_i}\in dx_i,i=1,\dots,n)&=&\frac{f(t_n,x_n;s,x)}{f(t_0,a;s,x)}\prod\limits_{i=1}^np_x(t_{i-1},x_{i-1};t_i,x_i)dx_i,\\
&=&\prod\limits_{i=1}^n\frac{f(t_i,x_i;s,x)}{f(t_{i-1},x_{i-1};s,x)}p_x(t_{i-1},x_{i-1};t_i,x_i)dx_i,\
\end{eqnarray*}
where $0=t_0<t_1<\cdots<t_n<s$. In our case, $\tilde{Y}_0=a$ and $\tilde{Y}_s=0$. Thus
\begin{eqnarray}
\nonumber\tilde{\mathbb{P}}^{a,0;s}(\tilde{Y}_{t_i}\in dx_i,i=1,\dots,n)&=&\frac{f(t_n,x_n;s,0)}{f(t_0,x;s,0)}\prod\limits_{i=1}^np_0(t_{i-1},x_{i-1};t_i,x_i)dx_i,\\
&=&\prod\limits_{i=1}^n\frac{f(t_i,x_i;s,0)}{f(t_{i-1},x_{i-1};s,0)}p_0(t_{i-1},x_{i-1};t_i,x_i)dx_i.\label{markov}
\end{eqnarray}
What we will show in the next section is that in fact this transition probability corresponds to that of the following diffusion
\begin{equation}\label{ybridge}
d\tilde{Y}_t:=h(t)dW_t+h^2(t)\left(\frac{1}{\tilde{Y}_t}-\frac{\tilde{Y}_t}{\int_t^sh^2(u)du}\right)dt\quad \tilde{Y}_0=a,\enskip t\in[0,s].
\end{equation}
\section{Verification of the dynamics of $\tilde{Y}$}\label{4}
We have answered the question: What is the transition probability of $\tilde{M}$, given that it hits the level $a$ for the first time at exactly time $s$. In this section we will verify  that the following identiy holds  $(\tilde{M}_t|T=s)=a-\tilde{Y}_t$ under $\tilde{\mathbb{P}}$ given that $\tilde{Y}$ is defined in (\ref{ybridge}). 

To do so, let us first find the transition probability of $\tilde{Y}$ using once more a p.d.e approach and then verifying that it equals  equation (\ref{markov}). That is, we must solve both the backward and forward Kolmogorov equations and  then construct the corresponding Green's function, {\it i.e.\/}
\begin{eqnarray}
\nonumber-\frac{\partial u'}{\partial t}(t,x)&=&\frac{1}{2}h^2(t)\frac{\partial^2u'}{\partial x^2}(t,x)
\\ \label{eq3}&&+h^2(t)\left(\frac{1}{x}-\frac{x}{\int_t^sh^2(u)du}\right)\frac{\partial u'}{\partial x}(t,x)\quad [0,\tau)\times(0,\infty)\\
\nonumber\frac{\partial v'}{\partial \tau}(\tau,y)&=&\frac{1}{2}h^2(\tau)\frac{\partial^2 v'}{\partial y^2}(\tau,y)\\
\label{eq31}&&-h^2(\tau)\left(\frac{1}{y}-\frac{y}{\int_\tau^sh^2(u)du}\right)\frac{\partial v'}{\partial y}(\tau,y)\\
\nonumber&&+h^2(\tau)\left(\frac{1}{y^2}-\frac{1}{\int_\tau^sh^2(u)du}\right)v'(\tau,y)\qquad (t,s]\times(0,\infty)
\end{eqnarray}
Although we will only derive the backward equation in this section (the reader  may consult the derivation of the forward equation at Appendix \ref{app}) the idea is to reduce equations (\ref{eq3}) and (\ref{eq31}) to the backward and forward heat equation respectively through algebraic transformations.

\begin{proposition}\label{pp1} [The techniques used in the proof of this Proposition, are very much in the spirit of Shoji and Takamizawa (2003).]
 The solution of the backward equation in (\ref{eq3}) is given by:
\begin{eqnarray*}
u'(t,x)=u_1(t,x)B(t)\exp\left\{A(t)x^2\right\}
\end{eqnarray*}
where
\begin{eqnarray}
\nonumber A(t)=\frac{1}{2\int_t^sh^2(u)du}&& \\
\nonumber B(t)=c\left(\int_t^sh^2(u)du\right)^{3/2}&&  \\
\label{backden1} (u_1)_t(t,x)+\frac{1}{2}h^2(t)(u_1)_{xx}(t,x)+h^2(t)\frac{1}{x}(u_1)_x(t,x)=0&&
\end{eqnarray}
for $x\in(0,\infty)$ and $t\in[0,\tau)$.
\end{proposition}
\begin{proof}
 We will first compute the partial derivatives. If we let the subscripts represent derivation with respect to the corresponding variables, then
\begin{eqnarray*}
u'_x(t,x)&=&(u_1)_xB(t)\exp\left\{A(t)x^2\right\}+2xA(t)u_1\cdot B(t)\exp\left\{A(t)x^2\right\}\\
u'_{xx}(t,x)&=&(u_1)_{xx} B(t)\exp\left\{A(t)x^2\right\}+4xA(t)(u_1)_xB(t)\exp\left\{A(t)x^2\right\}\\
&&+2A(t)u_1\cdot B(t)\exp\left\{A(t)x^2\right\}\\
&&+4x^2A^2(t)u_1\cdot B(t)\exp\left\{A(t)x^2\right\}\\
u'_t(t,x)&=&(u_1)_tB(t)\exp\left\{A(t)x^2\right\}+u_1\cdot B_t\exp\left\{A(t)x^2\right\}\\
&&+x^2A_t u_1\cdot B(t)\exp\left\{A(t)x^2\right\}.
\end{eqnarray*}
Substituting in (\ref{eq3}) we have
\begin{eqnarray*}
-\frac{(u_1)_t}{u_1}u'-\frac{B_t}{B}u'-x^2A_tu'&=&\frac{1}{2}h^2\frac{(u_1)_{xx}}{u_1}u'+2h^2xA\frac{(u_1)_x}{u_1}u'\\
&&+h^2A(t)u'+2h^2x^2A^2u'\\
&&+h^2\left[\frac{1}{x}-\frac{x}{\int_t^sh^2(u)du}\right]\frac{(u_1)_x}{u_1}u'\\
&&+h^2\left[\frac{1}{x}-\frac{x}{\int_t^sh^2(u)du}\right]2xA u',
\end{eqnarray*}
cancelling $u'$
\begin{eqnarray*}
-\frac{(u_1)_t}{u_1}-\frac{B_t}{B}-x^2A_t&=&\frac{1}{2}h^2\frac{(u_1)_{xx}}{u_1}+2h^2xA\frac{(u_1)_x}{u_1}\\
&&+h^2A(t)+2h^2x^2A^2\\
&&+h^2\left[\frac{1}{x}-\frac{x}{\int_t^sh^2(u)du}\right]\frac{(u_1)_x}{u_1}\\
&&+h^2\left[\frac{1}{x}-\frac{x}{\int_t^sh^2(u)du}\right]2xA, 
\end{eqnarray*}
collecting terms
\begin{eqnarray}
\label{fpde} 0&=&x^2\left[A_t+2h^2A^2-2A\frac{h^2}{\int_t^sh^2(u)du}\right]\\
\nonumber&&+h^2x\frac{(u_1)_x}{u_1}\left[2A-\frac{1}{\int_t^sh^2(u)du}\right]\\
\nonumber&&+\frac{(u_1)_t}{u_1}+\frac{1}{2}h^2\frac{(u_1)_{xx}}{u_1}+h^2\frac{1}{x}\frac{(u_1)_x}{u_1}\\
\nonumber&&+\left[\frac{B_t}{B}+3h^2 A\right].
\end{eqnarray}
We first observe from the second line in equation (\ref{fpde}), that:
\begin{eqnarray}\label{abc}
A(t)=\frac{1}{2\int_t^sh^2(u)du}.
\end{eqnarray}
Next, if we let $\tau:=\int_t^sh^2(u)du$ then  $d\tau/dt=-h^2(t)$. Hence, it follows from the last line in (\ref{fpde}) and equation (\ref{abc}) that

\begin{eqnarray*}
\frac{1}{B}\frac{dB}{dt}&=&-3\frac{h^2}{2\int_t^sh^2(u)du}\\
&=&\frac{3}{2}\frac{1}{\tau}\frac{d\tau}{dt}
\end{eqnarray*}
or
\begin{equation*}
\frac{1}{B}dB=\frac{3}{2}\frac{1}{\tau}d\tau,
\end{equation*}
{\it i.e.\/}:
\begin{eqnarray*}
B(t)=c\left(\int_t^sh^2(u)du\right)^{3/2},
\end{eqnarray*}
where $c$ is some arbitrary constant. 
The reader may verify that indeed the first term in (\ref{fpde}) vanishes. Alternatively, this will lead to:
\begin{eqnarray*}
(u_1)_t(t,x)+\frac{1}{2}h^2(t)(u_1)_{xx}(t,x)+h^2(t)\frac{1}{x}(u_1)_x(t,x)=0
\end{eqnarray*}
as claimed.
\end{proof}
Our next goal is to find a general solution to equation (\ref{backden1}) in Proposition \ref{pp1} and hence obtain:
\begin{proposition}
The transition probability of the process $\tilde{Y}$, defined in (\ref{ybridge}), equals
\begin{equation}\label{green11}
\tilde{G}(t,x;\tau,y)=\frac{f(\tau,a;s,0)}{f(t,a;s,0)}H(t,x;\tau,y)
\end{equation}
where $f$ is defined in equation (\ref{nose}) in Section \ref{3} and
\begin{eqnarray*}
\!\!\!&&H(t,x;\tau,y)\\
&&\quad=\frac{1}{\sqrt{2\pi\left[\int_t^\tau h^2(u)du\right]}}\left(\exp\left\{-\frac{(y-x)^2}{2\left[\int^\tau_th^2(u)du\right]}\right\}-\exp\left\{-\frac{(y+x)^2}{2\left[\int^\tau_th^2(u)du\right]}\right\}\right),
\end{eqnarray*}
[for $x,y>0$ and $0\leq t<\tau<s$]. In  particular $H(t,x;\tau,y)=p_0(t,x;\tau,y)$ where $p_0$ is defined in equation (\ref{nose1}) in Section \ref{3}.
\end{proposition}
\begin{proof}
 To do so we will introduce once more a time change $\tau':=\int_0^t h^2(u)du$ which implies that $d\tau'/dt=h^2(t)$. Next, substituting in equation (\ref{backden1}) it follows that
\begin{eqnarray*}
(u_1)_{\tau'}(\tau',x)+\frac{1}{2}(u_1)_{xx}(\tau',x)+\frac{1}{x}(u_1)_x(\tau',x)=0.
\end{eqnarray*}
Finally, set $u_1=1/x\cdot u_2$ to obtain the backward heat equation
\begin{eqnarray*}
(u_2)_{\tau'}(\tau',x)+\frac{1}{2}(u_2)_{xx}(\tau',x)=0.
\end{eqnarray*}
It is clear, see Appendix \ref{app}, at this point that the {\it reduced\/} forward Kolmogorov equation $v_3$ will satisfy alternatively:
\begin{eqnarray*}
-(v_3)_{\sigma'}(\sigma',y)+\frac{1}{2}(v_3)_{yy}(\sigma',y)=0
\end{eqnarray*}
with the constraint that both $x$ and $y$ are positive. Hence
\begin{equation*}
H(\tau',x;\sigma',y)=\frac{1}{\sqrt{2\pi(\sigma'-\tau')}}\left(\exp\left\{-\frac{(y-x)^2}{2(\sigma'-\tau')}\right\}-\exp\left\{-\frac{(y+x)^2}{2(\sigma'-\tau')}\right\}\right)
\end{equation*}
or with respect to the backward $(t,x)$ and the forward variables $(T,y)$ as:
\begin{eqnarray*}
&&H(t,x;T,y)\\
&&\quad=\frac{1}{\sqrt{2\pi(\int_t^Th^2(u)du)}}\left(\exp\left\{-\frac{(y-x)^2}{2(\int^T_th^2(u)du)}\right\}-\exp\left\{-\frac{(y+x)^2}{2(\int^T_th^2(u)du)}\right\}\right).
\end{eqnarray*}
This in turn leads to transition density of the process $\tilde{Y}$
\begin{equation*}
\tilde{G}(t,x;T,y)=\frac{\varphi_y\left[\int_T^sh^2(u)du\right]}{\varphi_x\left[\int_t^sh^2(u)du\right]}H(t,x;T,y)
\end{equation*}
as claimed.
\end{proof}

It follows that equations (\ref{markov}) and (\ref{green11}) are equivalent. The reader may find equivalent arguments in for instance: Imhof (1984) or  Revuz and Yor (2005).
\section{Proof of Theorem \ref{prop1}}\label{5}
\subsection{Change of measure} Recall that if $M$ is the solution of the following stochastic differential equation 
\begin{eqnarray*}
dM_s&=&h(s)d\tilde{B}_s+b(s)ds,\qquad M_0=x,
\end{eqnarray*}
and
\begin{eqnarray*}
\int_0^t\left[\frac{b(u)}{h(u)}\right]^2du<+\infty\qquad t\in[0,+\infty).
\end{eqnarray*}
Then under the measure $\hat{\mathbb{P}}$ induced by
\begin{eqnarray*}
\tilde{Z}_s&:=&\exp\left\{-\int_0^s\frac{b(u)}{h(u)}d\tilde{B}_u-\frac{1}{2}\int_0^s\left(\frac{b(u)}{h(u)}\right)^2du\right\}\qquad\hbox{or}\\
&=&\exp\left\{-\int_0^s\frac{b(u)}{h^2(u)}d\tilde{M}_u-\frac{1}{2}\int_0^s\left(\frac{b(u)}{h(u)}\right)^2du\right\}
\end{eqnarray*}
the process $\tilde{M}$ is a $\hat{\mathbb{P}}$-martingale.




\begin{proof}[Proof of Theorem \ref{prop1}] Given the following continuous $\mathbb{P}$-martingale $M$  
\begin{equation*}
M_t:=\int_0^th(u)dB_u\qquad M_0=0
\end{equation*}
and letting
\begin{eqnarray*}
T&:=&\inf\{t\geq 0|M_t=f(t)\}\\
&=&\inf\{t\geq 0|M_t-\int_0^tf'(u)du=f(0)\}\\
&=&\inf\{t\geq 0|\tilde{M}_t=a\}\quad\hbox{under $\tilde{\mathbb{P}}$},
\end{eqnarray*}
where the Radon-Nikodym derivative equals $d\tilde{\mathbb{P}}/d\mathbb{P}=Z$ and
\begin{eqnarray*}
\tilde{Z}_t
&=&\exp\left\{-\int_0^t\frac{f'(u)}{h(u)}d\tilde{B}_u-\frac{1}{2}\int_0^t\left[\frac{f'(u)}{h(u)}\right]^2du\right\}\\
&=&\exp\left\{-\int_0^t\frac{f'(u)}{h^2(u)}d\tilde{M}_u-\frac{1}{2}\int_0^t\left[\frac{f'(u)}{h(u)}\right]^2du\right\}\\
&=&\exp\left\{-\int_0^s\beta(u)d\tilde{M}_u-\frac{1}{2}\int_0^s\beta(u)f'(u)du\right\}\\
\end{eqnarray*}
where $\beta(u)=f'(u)/h^2(u)$ for $u\geq 0$. Then
\begin{eqnarray*}
\mathbb{P}(T<t)&=&\tilde{\mathbb{E}}\left[\exp\left\{-\int_0^t\beta(u)d\tilde{M}_u-\frac{1}{2}\int_0^t\beta(u)f'(u)du\right\}\mathbb{I}_{(T<t)}\right]\\
&=&\tilde{\mathbb{E}}\left[\exp\left\{-\beta(t)\tilde{M}_t+\int_0^t\beta'(u)\tilde{M}_udu-\frac{1}{2}\int_0^t\beta(u)f'(u)du\right\}\mathbb{I}_{(T<t)}\right]\\
&=&\tilde{\mathbb{E}}\left[\exp\left\{-\beta(T)\tilde{M}_T+\int_0^T\beta'(u)\tilde{M}_udu-\frac{1}{2}\int_0^T\beta(u)f'(u)du\right\}\mathbb{I}_{(T<t)}\right]\\
&=&\int_0^te^{-\beta(s)a-\frac{1}{2}\int_0^s\beta(u)f'(u)du}\tilde{\mathbb{E}}\left[\exp\left\{\int_0^s\beta'(u)\tilde{M}_udu\right\}\Bigg{|}T=s\right]\varphi_a(s)ds
\end{eqnarray*}
Finally, from equation (\ref{ybridge}) in Section \ref{3} we have:
\begin{eqnarray*}
&&\tilde{\mathbb{E}}\left[\exp\left\{\int_0^s\beta'(u)\tilde{M}_udu\right\}\Bigg{|}T=s\right]\\
&&\qquad\quad=\exp\left\{a\int_0^s\beta'(u)du\right\}\tilde{\mathbb{E}}\left[\exp\left\{-\int_0^s\beta'(u)\tilde{Y}_udu\right\}\right].
\end{eqnarray*}
\end{proof}


\section{First passage time}\label{6}
This section is devoted to the computation of the following expectation
\begin{eqnarray*}
\tilde{\mathbb{E}}\left[\exp\left\{-\int_0^s\beta'(u)\tilde{Y}_udu)\right\}\right],
\end{eqnarray*}
which is equivalent to solving the following Cauchy problem:
\begin{eqnarray*}
-\frac{\partial u}{\partial t}(t,a)&=&-\beta'(t)a u(t,a)+\frac{1}{2}h^2(t)\frac{\partial u}{\partial a^2}(t,a)\\
&&+h^2(t)\left(\frac{1}{a}-\frac{a}{\int_t^sh^2(u)du}\right)\frac{\partial u}{\partial a}(t,a),\quad[0,s)\times(0,\infty)\\
v(s,a)&=&1.
\end{eqnarray*}

The idea  is to solve simultaneously both the backward and forward Kolmogorov equations:
\begin{eqnarray}
\nonumber-\frac{\partial \varphi}{\partial t}(t,a)&=&-\beta'(t)a \varphi(t,a)+\frac{1}{2}h^2(t)\frac{\partial \varphi}{\partial a^2}(t,a)\\
\label{b1}&&+h^2(t)\left(\frac{1}{a}-\frac{a}{\int_t^sh^2(u)du}\right)\frac{\partial \varphi}{\partial a}(t,a),\quad[0,\tau)\times(0,\infty)\\
\nonumber\frac{\partial \psi}{\partial \tau}(\tau,b)&=&-\beta'(\tau)b \psi(\tau,b)+\frac{1}{2}h^2(\tau)\frac{\partial \psi}{\partial b^2}(\tau,b)\\
\label{b2}&&+h^2(\tau)\frac{\partial}{\partial b}\left\{\left(\frac{1}{b}-\frac{b}{\int_\tau^sh^2(u)du}\right)\psi(\tau,b)\right\},\quad(t,s]\times(0,\infty)
\end{eqnarray}
in order to obtain its corresponding Green's function. 
\begin{proposition}
The backward and forward equations in (\ref{b1}) and (\ref{b2}) satisfy the following relationships respectively for $0\leq t<\tau\leq s$ and $a,b\in(0,\infty)$:
\begin{eqnarray*}
\varphi(t,a)&=&\varphi^1(t,a)\frac{B(t)}{a}\exp\left\{A(t)a^2\right\}\\
\psi(\tau,b)&=&\psi^1(\tau,b)b\tilde{B}(\tau)\exp\left\{\tilde{A}(\tau)b^2\right\}
\end{eqnarray*}
where
\begin{eqnarray*}
&&A(t)=\frac{1}{2\int_t^sh^2(u)du}\qquad\quad B(t)=c\left(\int_t^sh^2(u)du\right)^{3/2}
\\
&&\tilde{A}(\tau)=-\frac{1}{s\int_\tau^sh^2(u)du}\qquad
\tilde{B}(\tau)=c_1\left(\int_t^sh^2(u)du\right)^{-3/2}
\end{eqnarray*}
and
\begin{eqnarray*}
-\frac{\partial\varphi^1}{\partial t}(t,a)+\beta'(t)a\varphi^1(t,a)&=&\frac{1}{2}h^2(t)\frac{\partial\varphi^1}{\partial a^2}(t,a)\\
\frac{\partial\psi^1}{\partial \tau}(\tau,b)+\beta'(\tau)b\psi^1(\tau,b)&=&\frac{1}{2}h^2(\tau)\frac{\partial\psi^1}{\partial b^2}(\tau,b).
\end{eqnarray*}
\end{proposition}
\begin{proof} The reader may derive it from the techniques used in Section \ref{4} and Appendix \ref{app}.
\end{proof}

The backward and forward Schr\"odinger equations have the following solution:

\begin{theorem} For $0\leq t<\tau\leq s$, and $a,b\in(0,\infty)$, and given that
\begin{eqnarray*}
\tilde{G}(t',y;\tau',z):=\frac{1}{\sqrt{2\pi(\tau'-t')}}\left[\exp\left\{-\frac{(z-y)^2}{2(\tau'-t')}\right\}-\exp\left\{-\frac{(z+y)^2}{2(\tau'-t')}\right\}\right]
\end{eqnarray*}
where $\tau'-t'=\int_t^\tau h^2(u)du$, $y=a+v(t)$, $z=b+\tilde{v}(\tau)$ are obtained by solving:
\begin{eqnarray*}
&&\pi_t(t)=\beta'(t)\enskip\qquad\quad
v'(t)=-h^2(t)\pi(t)\\
&&\tilde{\pi}_\tau(\tau)=-\tilde{\beta}'(\tau)\qquad \tilde{v}'(\tau)=h^2(\tau)\tilde{\pi}(\tau).
\end{eqnarray*} 
Then the following Green's function $G$:
\begin{equation*}
G(t,a;\tau,z)= \exp\left\{b\tilde{\pi}(\tau)+a\pi(t)+\frac{1}{2}\int_t^\tau h^2(u)\pi^2(u)\right\}\tilde{G}(t',y;\tau',z)
\end{equation*}
satisfies both the backward and forward Schr\"odinger's equations for time-dependent linear potential:
\begin{eqnarray*}
-\frac{\partial\varphi^1}{\partial t}(t,a)+\beta'(t)a\varphi^1(t,a)&=&\frac{1}{2}h^2(t)\frac{\partial\varphi^1}{\partial a^2}(t,a)\\
\frac{\partial\psi^1}{\partial \tau}(\tau,b)+\beta'(\tau)b\psi^1(\tau,b)&=&\frac{1}{2}h^2(\tau)\frac{\partial\psi^1}{\partial b^2}(\tau,b).
\end{eqnarray*}
\end{theorem}
 
\begin{proof} Let $\varphi^1(t,a)=\lambda(t,a)e^{\pi(t)a}$ and $\psi^1(\tau,b)=\tilde{\lambda}(\tau,b)e^{\tilde{\pi}(\tau)b}$, with $\pi(t)$ and $\tilde{\pi}(\tau)$ being  time dependent variables determined later, we have
\begin{eqnarray*}
\varphi^1_t(t,a)&=&\lambda_t(t,a)e^{\pi(t)a}+a\pi_t(t)\lambda(t,a) e^{\pi(t)a}\\
\varphi^1_a(t,a)&=&\lambda_a(t,a)e^{\pi(t)a}+\pi(t)\lambda(t,a) e^{\pi(t)a}\\
\varphi^1_{aa}(t,a)&=&\lambda_{aa}(t,a)e^{\pi(t)a}+2\pi(t)\lambda_{a}(t,a)e^{\pi(t)a}+\pi^2(t)\lambda(t,a) e^{\pi(t)a}\\
\end{eqnarray*}
and
\begin{eqnarray*}
\psi^1_\tau(\tau,b)&=&\tilde{\lambda}_\tau(\tau,b)e^{\tilde{\pi}(\tau)b}+b\tilde{\pi}_\tau(\tau)\tilde{\lambda}(\tau,b) e^{\tilde{\pi}(\tau)b}\\
\psi^1_b(\tau,b)&=&\tilde{\lambda}_b(\tau,b)e^{\tilde{\pi}(\tau)b}+\tilde{\pi}(\tau)\tilde{\lambda}(\tau,b) e^{\tilde{\pi}(\tau)b}\\
\psi^1_{bb}(\tau,b)&=&\lambda_{bb}(\tau,b)e^{\tilde{\pi}(\tau)b}+2\tilde{\pi}(\tau)\tilde{\lambda}_{b}(\tau,b)e^{\tilde{\pi}(\tau)b}+\tilde{\pi}^2(\tau)\tilde{\lambda}(\tau,b) e^{\tilde{\pi}(\tau)b}\\
\end{eqnarray*}
which after substitution
\begin{eqnarray*}
-\lambda_t(t,a)-a\pi_t(t)\lambda(t,a)+a\beta'(t)\lambda(t,a)&=&\frac{1}{2}h^2(t)\lambda_{aa}(t,a)+h^2(t)\pi(t)\lambda_a(t,a)\\
&&+\frac{1}{2}h^2(t)\pi^2(t)\lambda(t,a).
\end{eqnarray*}
and
\begin{eqnarray*}
\tilde{\lambda}_\tau(\tau,b)+b\tilde{\pi}_\tau(\tau)\tilde{\lambda}(\tau,b)+b\tilde{\beta}'(\tau)\tilde{\lambda}(\tau,b)&=&\frac{1}{2}h^2(\tau)\tilde{\lambda}_{bb}(\tau,b)+h^2(\tau)\tilde{\pi}(\tau)\tilde{\lambda}_b(\tau,b)\\
&&+\frac{1}{2}h^2(\tau)\tilde{\pi}^2(\tau)\tilde{\lambda}(\tau,b).
\end{eqnarray*}

If we perform the time and space transformations $y=a+v(t)$ and $z=b+\tilde{v}(\tau)$  where $v(t)$ and $\tilde{v}(\tau)$ will be determined later, and set $\lambda(t,a)=\epsilon(t,y)$ and $\tilde{\lambda}(\tau,b)$, {\it i.e.\/}
\begin{eqnarray*}
\lambda_t(t,a)&=&\epsilon_{t}(t,y)+v'(t)\epsilon_y(t,y)\\
\lambda_a(t,a)&=&\epsilon_y(t,y)\\
\lambda_{aa}(t,a)&=&\epsilon_{yy}(t,y),
\end{eqnarray*}
and
\begin{eqnarray*}
\tilde{\lambda}_\tau(\tau,b)&=&\tilde{\epsilon}_{\tau}(\tau,z)+\tilde{v}'(\tau)\tilde{\epsilon}_z(\tau,z)\\
\tilde{\lambda}_b(\tau,b)&=&\tilde{\epsilon}_z(\tau,z)\\
\tilde{\lambda}_{bb}(\tau,b)&=&\tilde{\epsilon}_{zz}(\tau,z),
\end{eqnarray*}
then
\begin{eqnarray*}
-\epsilon_{t}(t,y)+\left[y-v(t)\right]\left[\beta'(t)-\pi_t(t)\right]\epsilon(t,y)&=&\frac{1}{2}h^2(t)\epsilon_{yy}(t,y)\\
&&+\left[v'(t)+h^2(t)\pi(t)\right]\epsilon_y(t,y)\\
&&+\frac{1}{2}h^2(t)\pi^2(t)\epsilon(t,y)
\end{eqnarray*}
and
\begin{eqnarray*}
\tilde{\epsilon}_{\tau}(\tau,z)+\left[z-\tilde{v}(\tau)\right]\left[\tilde{\beta}'(\tau)+\tilde{\pi}_\tau(\tau)\right]\tilde{\epsilon}(\tau,z)&=&\frac{1}{2}h^2(\tau)\tilde{\epsilon}_{zz}(\tau,z)\\
&&+\left[h^2(\tau)\tilde{\pi}(\tau)-\tilde{v}'(\tau)\right]\tilde{\epsilon}_z(\tau,z)\\
&&+\frac{1}{2}h^2(\tau)\tilde{\pi}^2(\tau)\tilde{\epsilon}(\tau,z)
\end{eqnarray*}
and setting
\begin{eqnarray*}
&&\pi_t(t)=\beta'(t)\enskip\qquad\quad
v'(t)=-h^2(t)\pi(t)\\
&&\tilde{\pi}_\tau(\tau)=-\tilde{\beta}'(\tau)\qquad \tilde{v}'(\tau)=h^2(\tau)\tilde{\pi}(\tau)
\end{eqnarray*}
we have
\begin{eqnarray*}
-\epsilon_t(t,y)&=&\frac{1}{2}h^2(t)\epsilon_{yy}(t,y)+\frac{1}{2}h^2(t)\pi^2(t)\epsilon(t,y)\\
\tilde{\epsilon}_\tau(\tau,z)&=&\frac{1}{2}h^2(\tau)\tilde{\epsilon}_{zz}(\tau,z)+\frac{1}{2}h^2(\tau)\tilde{\pi}^2(\tau)\tilde{\epsilon}(\tau,z),
\end{eqnarray*}
next, let
\begin{eqnarray*}
\epsilon(t,y)&=&\varphi^2(t,y)\exp\left\{-\frac{1}{2}\int_0^th^2(u)\pi^2(u)du\right\}\\
\tilde{\epsilon}(\tau,z)&=&\psi^2(\tau,z)\exp\left\{\frac{1}{2}\int_0^\tau h^2(u)\tilde{\pi}^2(u)du\right\},
\end{eqnarray*}
thus
\begin{eqnarray*}
-\varphi^2_t(t,y)&=&\frac{1}{2}h^2(t)\varphi^2_{yy}(t,y),\\
\psi^2_\tau(\tau,z)&=&\frac{1}{2}h^2(\tau)\psi^2_{zz}(\tau,z).
\end{eqnarray*}
Finally, introduce the following  time changes: $t'=\int_0^th^2(u)du$ and $\tau'=\int_0^\tau h^2(u)du$
to obtain
\begin{eqnarray*}
-\varphi^2_{t'}(t',y)&=&\frac{1}{2}\varphi^2_{yy}(t',y)\\
\psi^2_{\tau'}(\tau',z)&=&\frac{1}{2}\varphi^2_{zz}(\tau',z).
\end{eqnarray*}
We now start by constructing a function $\tilde{G}$ which satisfies both equations and the constraint that $a,b\in(0,\infty)$, {\it i.e.\/}
\begin{equation*}
\tilde{G}(t',y;\tau',z)=\frac{1}{\sqrt{2\pi(\tau'-t')}}\left[\exp\left\{-\frac{(z-y)^2}{2(\tau'-t')}\right\}-\exp\left\{-\frac{(z+y)^2}{2(\tau'-t')}\right\}\right]
\end{equation*}
\end{proof}
\section{Appendix}\label{app}
\begin{proposition}
 The solution of the forward equation in (\ref{eq31}) is given by:
\begin{eqnarray*}
v'(\tau,y)=v_1(\tau,y)B(\tau)\exp\left\{A(\tau)y^2\right\}
\end{eqnarray*}
where
\begin{eqnarray*}
\nonumber A(\tau)=-\frac{1}{2\int_\tau^sh^2(u)du}&& \\
\nonumber B(t)=c\left(\int_\tau^sh^2(u)du\right)^{-1/2} 
\end{eqnarray*}
and
\begin{eqnarray}
\label{backden11} (v_1)_\tau(\tau,y)&=&\frac{1}{2}h^2(\tau)(v_1)_{yy}(\tau,y)-h^2(\tau)\frac{1}{y}(v_1)_y(\tau,y)\\
\nonumber &&+h^2(\tau)\left[\frac{1}{y^2}-\frac{1}{\int_\tau^sh^2(u)du}\right]v_1(\tau,y)
\end{eqnarray}
for $y\in(0,\infty)$ and $\tau\in(t,s]$.
\end{proposition}
\begin{proof}
 We will first compute the partial derivatives. If the subscripts represent derivation with respect to the corresponding variables, we have
\begin{eqnarray*}
v'_y(\tau,y)&=&(v_1)_yB(\tau)\exp\left\{A(\tau)y^2\right\}+2yA(\tau)v_1\cdot B(T)\exp\left\{A(\tau)y^2\right\}\\
v'_{yy}(\tau,y)&=&(v_1)_{yy} B(\tau)\exp\left\{A(\tau)y^2\right\}+4yA(\tau)(v_1)_yB(\tau)\exp\left\{A(\tau)y^2\right\}\\
&&+2A(\tau)v_1\cdot B(\tau)\exp\left\{A(\tau)y^2\right\}\\
&&+4y^2A^2(\tau)v_1\cdot B(\tau)\exp\left\{A(\tau)y^2\right\}\\
v'_\tau(\tau,y)&=&(v_1)_\tau B(\tau)\exp\left\{A(\tau)x^2\right\}+v_1\cdot B_\tau\exp\left\{A(\tau)y^2\right\}\\
&&+y^2A_\tau v_1\cdot B(\tau)\exp\left\{A(\tau)y^2\right\}
\end{eqnarray*}
Substituting in (\ref{eq31}) we have
\begin{eqnarray*}
\frac{(v_1)_\tau}{(v_1)}v'+\frac{B_\tau}{B}v'+y^2A_\tau v'&=&\frac{1}{2}h^2\frac{(v_1)_{yy}}{v_1}v'+2h^2yA\frac{(v_1)_y}{v_1}v'\\
&&+h^2A(\tau)v'+2h^2y^2A^2v'\\
&&-h^2\left[\frac{1}{y}-\frac{y}{\int_\tau^sh^2(u)du}\right]\frac{(v_1)_y}{v_1}v'\\
&&-h^2\left[\frac{1}{y}-\frac{y}{\int_\tau^sh^2(u)du}\right]2yA v',\\
&&+h^2\left[\frac{1}{y^2}-\frac{1}{\int_\tau^sh^2(u)du}\right]v'
\end{eqnarray*}
cancelling $v'$
\begin{eqnarray*}
\frac{(v_1)_\tau}{v_1}+\frac{B_\tau}{B}+y^2A_\tau &=&\frac{1}{2}h^2\frac{(v_1)_{yy}}{v_1}+2h^2yA\frac{(v_1)_y}{v_1}\\
&&+h^2A(\tau)+2h^2y^2A^2\\
&&-h^2\left[\frac{1}{y}-\frac{y}{\int_\tau^sh^2(u)du}\right]\frac{(v_1)_y}{v_1}\\
&&-h^2\left[\frac{1}{y}-\frac{y}{\int_\tau^sh^2(u)du}\right]2yA ,\\
&&+h^2\left[\frac{1}{y^2}-\frac{1}{\int_\tau^sh^2(u)du}\right]
\end{eqnarray*}
collecting terms
\begin{eqnarray*}
0&=&y^2\left[2h^2A^2-A_\tau+2A\frac{h^2}{\int^s_\tau h^2(u)du}\right]\\
&&+h^2y\frac{(v_1)_y}{v_1}\left[2A+\frac{1}{\int_\tau^sh^2(u)du}\right]\\
&&-\frac{(v_1)_\tau}{v_1}+\frac{1}{2}h^2\frac{(v_1)_{yy}}{v_1}-h^2\frac{1}{y}\frac{(v_1)_y}{v_1}+h^2\left[\frac{1}{y^2}-\frac{1}{\int_\tau^sh^2(u)du}\right]\\
&&-\frac{B_\tau}{B}-h^2 A\\
\end{eqnarray*}
We first observe that:
\begin{eqnarray*}
A(\tau)=-\frac{1}{2\int_\tau^sh^2(u)du}.
\end{eqnarray*}
Furthermore letting: $\sigma=\int_t^\tau h^2(u)du$, and observing that  $d\sigma/d\tau=-h^2(\tau)$ we have that
\begin{eqnarray*}
\frac{1}{B}\frac{dB}{d\tau}&=&-h^2\left[-\frac{1}{2\int_t^sh^2(u)du}\right]\\
&=&-\frac{1}{2}\frac{1}{\sigma}\frac{d\sigma}{d\tau}
\end{eqnarray*}
or
\begin{equation*}
\frac{1}{B}dB=-\frac{1}{2}\frac{1}{\sigma}d\sigma
\end{equation*}
which implies that
\begin{eqnarray*}
B(\tau)=c\left(\int_\tau^s h^2(u)du\right)^{-1/2},
\end{eqnarray*}
and
\begin{eqnarray*}
(v_1)_\tau(\tau,y)&=&\frac{1}{2}h^2(\tau)(v_1)_{yy}(\tau,y)-h^2(\tau)\frac{1}{y}(v_1)_y(\tau,y)\\
&&+h^2(\tau)\left[\frac{1}{y^2}-\frac{1}{\int_\tau^sh^2(u)du}\right]v_1(\tau,y)
\end{eqnarray*}
as claimed.
\end{proof}
Now set $v_1=y\cdot v_2$, {\it i.e.\/}
\begin{eqnarray*}
(v_1)_{\tau}&=&y (v_2)_\tau\\
(v_1)_y&=&y (v_2)_y+v_2\\
(v_1)_{yy}&=&y (v_2)_{yy}+2 (v_2)_y
\end{eqnarray*}
and substituting into
\begin{eqnarray*}
y(v_2)_{\tau}&=&h^2(\tau)\frac{y}{2}(v_2)_{yy}+h^2(\tau)(v_2)_y\\
&&-h^2(\tau)(v_2)_y-\frac{1}{y}h^2(\tau)v_2\\
&&+\frac{1}{y}h^2(\tau)v_2-y\frac{h^2(\tau)}{\int_\tau^sh^2(u)du}v_2
\end{eqnarray*}
or
\begin{eqnarray*}
(v_2)_{\tau}&=&\frac{1}{2}h^2(\tau)(v_2)_{yy}-\frac{h^2(\tau)}{\int_\tau^sh^2(u)du}v_2
\end{eqnarray*}
Next set
\begin{eqnarray*}
v_2(\tau,y)=\frac{1}{\int_\tau^sh^2(u)du}v_3(\tau,y)
\end{eqnarray*}
Thus
\begin{eqnarray*}
(v_2)_\tau(\tau,y)=\frac{1}{\int_\tau^sh^2(u)du}(v_3)_\tau(\tau,y)-\frac{h^2(\tau)}{\left(\int_\tau^sh^2(u)du\right)^2}v_3(\tau,y)
\end{eqnarray*}
reducing the equation to
\begin{equation*}
(v_3)_\tau(\tau,y)=\frac{1}{2}h^2(\tau)(v_3)_{yy}(\tau,y)
\end{equation*}
Finally, given the following time change $\sigma':=\int_0^\tau h^2(u)du$ and observing that $d\sigma'/d\tau=h^2(\tau)$, substitute in the previous equation to obtain
\begin{equation*}
(v_3)_{\sigma'}(\sigma',y)=\frac{1}{2}(v_3)_{yy}(\sigma',y).
\end{equation*}

\end{document}